\documentclass[12pt]{article}
\usepackage[english]{babel}
\usepackage[T1]{fontenc}
\usepackage[utf8]{inputenc}
\usepackage[a4paper]{geometry}
\usepackage{bm}
\usepackage{dsfont}
\usepackage{amssymb}
\usepackage{amsmath}
\usepackage{amsthm}
\usepackage{amsfonts}
\usepackage{mathrsfs}
\usepackage{mathtools,marvosym}

\usepackage[table]{xcolor}
\usepackage{enumitem}

\usepackage[flushmargin]{footmisc}

\RequirePackage[colorlinks=true,linkcolor=black,pdfpagelabels,citecolor=orange,urlcolor=teal]{hyperref}
\usepackage{url}

\newtheorem{theorem}{Theorem}
\newtheorem{lemma}[theorem]{Lemma}
\newtheorem{prop}[theorem]{Proposition}
\newtheorem{corollary}[theorem]{Corollary}
\theoremstyle{definition}

\newtheorem{remark}[]{Remark}

\newcommand{\N}{\mathbb{N}} 
\newcommand{\Z}{\mathbb{Z}} 
\newcommand{\R}{\mathbb{R}} 
\newcommand{\dd}{\mathop{}\!\mathrm{d}}
\renewcommand{\phi}{\varphi}
\renewcommand{\epsilon}{\varepsilon}
\renewcommand{\theta}{\vartheta}

\newcommand{\abs}[1]{\left\lvert#1\right\rvert} 
\newcommand{\norm}[1]{\left\lVert #1\right\rVert} 
\newcommand{\pscl}[2]{\left< #1,#2\right>}
\newcommand{\Id}{\mathbb{I}} 
\DeclareMathOperator{\rot}{rot} 
\DeclareMathOperator{\supp}{supp}  
\DeclareMathOperator{\cl}{cl}  
\DeclareMathOperator{\GRot}{\overline{\rot}}
\DeclareMathOperator{\interior}{int}
\newcommand{\Null}{\mathcal{N}}
\newcommand{\DD}{\mathcal{D}} 
\DeclareMathOperator{\halfPlane}{\mathcal{H}}
\DeclareMathOperator{\dom}{dom} 
\newcommand{\rr}{\mathfrak{r}}

\numberwithin{equation}{section}
\numberwithin{theorem}{section}
\numberwithin{remark}{section}

\title{Existence of a periodic solution for superquadratic Hamiltonian systems with possible finite-time blow-up}
\author{Alberto Cagnetta and Paolo Gidoni }

\date{}

\begin{document}

\maketitle
\begin{abstract}
	We prove a sufficient condition for the existence of a $T$-periodic solution for the planar system $\dot z=F(t,z)$, characterized by the growth to infinity of the rotations made in one period by solutions starting at increasingly large initial values. 	
	Our result applies in particular to superquadratic Hamiltonian systems satisfying the Ambrosetti--Rabinowitz condition. The key novelty of the paper is that we do not require any growth condition on the flow to ensure global existence of solutions, allowing finite-time blow-up. Our method is based on a fixed-point theorem which exploits the rotational properties of the dynamics. To conclude, we discuss a family of examples of Hamiltonian systems showing finite-time blow-up.
\end{abstract}

\noindent\textbf{MSC (2020).} Primary 34C25; Secondary 37C25, 37E45.

\noindent\textbf{Keywords.} Periodic solutions; superquadratic Hamiltonian systems; blow-up solution; Ambrosetti--Rabinowitz condition; rotation; fixed-point theorem. 

\footnotetext{Dipartimento Politecnico di Ingegneria e Architettura, Università degli Studi di Udine, Via delle Scienze 206, 33100 Udine, Italy. \Letter \; \texttt{alberto.cagnetta@uniud.it}, \texttt{paolo.gidoni@uniud.it}}

\section{Introduction and motivation}	

In this paper we discuss the existence of a $T$-periodic solution for a family of planar systems
\begin{equation}\label{eq:planar}
	\dot z=F(t,z)
\end{equation}
where $F$ is  $T$-periodic in $t$.

The characterizing feature imposed on \eqref{eq:planar} will be a rotational property of its solutions as $\abs{z}\to+\infty$, which can be roughly summarized by stating that solutions of \eqref{eq:planar} with sufficiently large $\abs{z(0)}$ make an arbitrarily large number of rotations around the origin during the time interval $[0,T]$. Such a property is pivotal, for instance, in the study of scalar superlinear second order ODEs and of planar superquadratic Hamiltonian systems. 

The investigation of this topic began with the study of the existence of a periodic solution for the second order superlinear ODE
\begin{equation}\label{eq:superlinearODE}
	\ddot x+f(t,x)=0
\end{equation}
where $f$ is $T$-periodic in $t$ and
\begin{equation}\label{eq:superlinearODEgrowth}
	\lim_{\abs{x}\to+\infty}\frac{f(t,x)}{x}=+\infty \qquad \text{uniformly in $t\in[0,T]$.}
\end{equation}
The problem was first studied by Morris \cite{Mor1,Mor2} for $f(t,x)=x^3+p(t)$ and by Ehrmann \cite{Ehr} for $f(t,x)=g(x)+p(t)$ under symmetry conditions. The general case $f(t,x)=g(x)+p(t)$ was solved by Fučík and Lovicar \cite{Fucik} through a fixed-point argument, which was later improved by Struwe \cite{Stru} to the case $\abs{f(t,x,\dot x)-g(x)}\leq K(\abs{x}+\abs{\dot x})+p(t)$, where $g$ grows superlinearly. We refer to \cite{CaMaZa} for a more detailed history of the problem, whereas we will soon discuss the recent contribution in \cite{Gid23}.

To exploit the rotational structure of the dynamics, equation~\eqref{eq:superlinearODE} is often studied on the phase plane $z=(z_1,z_2)=(x,\dot x)$, on which it assumes the  Hamiltonian structure
\begin{align}\label{eq:HS}
	\dot z(t) =J\nabla_{z}H(t,z)\,, && J=\begin{pmatrix}
	0&1\\-1&0
	\end{pmatrix} \,.
\end{align}
For \eqref{eq:superlinearODE}, more precisely, we have $H(t,z)=\frac{1}{2}z_2^2+\int_{0}^{z_1}f(t,x)\dd x$, meaning that the Hamiltonian function $H$ grows quadratically in $z_2$ and superquadratically in $z_1$. 
This structure suggested extending the investigation to include planar Hamiltonian systems where the Hamiltonian function $H(t,z)$ is $T$-periodic in $t$ and grows superquadratically in $\abs{z}$. This is usually achieved by imposing an \emph{Ambrosetti--Rabinowitz}-type condition: \emph{there exist $k\in(0,1/2)$, $m>0$ and $R>0$ such that}
\begin{equation}
	\label{eq:Ambrosetti-Rabinowitz}
	0<m\leq H(t,z)<k\pscl{z}{\nabla_{z}H(t,z)}\qquad\text{for every $\abs{z}>R$,}
\end{equation}
who indeed implies $H(t,z)>a\abs{z}^{\frac{1}{k}}$ for every $\abs{z}>R$, cf. Lemma~\ref{lem:growth_H_under_Ambrosetti-Rabinowitz}. The interest in the problem was first raised by Rabinowitz \cite{Rab78}, focusing first on the autonomous case and then briefly discussing the existence of a periodic solution in the nonautonomous case. A key contribution is due to Bahri and Berestycki \cite{BaBe}, which, with variational methods, found the existence of infinitely many $T$-periodic solutions for a $\mathcal{C}^2$-smooth Hamiltonian $H(t,z)=H_0(z)+p(t)\cdot z$, with $H_0$ satisfying \eqref{eq:Ambrosetti-Rabinowitz} and a growth-condition $a \abs{z}^{b+1}-c\leq H_0(z)\leq a'\abs{z}^{b'+1}+c'$, with $1<b\leq b'<2b+1$. Improvements of this result, under weaker conditions, can be found in \cite{Long90,PisTuc} employing variational methods, and in \cite{Bos12} using the Poincaré--Birkhoff Theorem.

While the Hamiltonian structure \eqref{eq:HS} allows also a variational approach, the more general formulation \eqref{eq:planar} as planar system can be studied with topological techniques based on the notion of rotation of a solution.  Given a solution $z(t)=\zeta(t;0,z_0)$ of the initial value problem  \begin{equation}\label{eq:cauchyIVP}
	\begin{cases}
		\dot{z}(t)=F(t,z)\,,\\
		z(0)=z_0\,,
	\end{cases}
\end{equation} existing on an interval $t\in[0,\bar t]$ and such that $z(t)\neq 0$ for every $t\in[0,\bar t]$, we define the clockwise \emph{rotation} of the solution during $[0,\bar t]$ as
\begin{equation*}
	\rot_{\bar t}(z_{0})= \frac{1}{2\pi}\int_{0}^{\bar t}\frac{\pscl{F(s,z(s))}{Jz(s)}}{\abs{z(s)}^{2}}\,\dd s \,.
\end{equation*}
As we discuss in Section~\ref{sec:prel}, rotation is, more properly, related to the lift in polar coordinates of the orbit of the solution, which explains why we do not define it for orbits crossing the origin $0_{\R^2}$.

In addition to some mild regularity assumptions providing local existence and uniqueness of solution, all the results in literature discussed above have in common the presence of two key assumptions:
\begin{enumerate}[label=\textup{\roman*)}]
\item \emph{global existence of solutions for positive times}: 
For the ODE \eqref{eq:superlinearODE} it is automatically satisfied via energy estimates if $f(t,x)=g(x)+p(t)$, while for example in \cite{Stru} it is obtained by the previously stated growth-condition on $f$. For the Hamiltonian system \eqref{eq:HS} in \cite{BaBe}, it is given by the upper growth-estimate on $H_0$.

\label{cond:global}
\item \emph{a rotational condition}: for example the superlinearity \eqref{eq:superlinearODEgrowth} for \eqref{eq:superlinearODE} or the Ambrosetti--Rabinowitz condition \eqref{eq:Ambrosetti-Rabinowitz} for \eqref{eq:HS}. \label{cond:rotational}
\end{enumerate}
A nice discussion and weaker versions of both these conditions for \eqref{eq:HS} can be found in \cite{Bos12}.

In the first part of the work, instead, we will adopt a more abstract approach, removing the Hamiltonian structure and thus extending the results to general planar systems \eqref{eq:planar}. In this more topological perspective, the rotational condition \ref{cond:rotational} assures two useful behaviours of the system.
\begin{itemize}	
	\item \emph{Growth at infinity of the rotation number}: namely, for every $n\in\N$ there exists a radius $R_n>0$ such that if $\abs{z_0}\geq R_n$ then $\rot_T(z_0)>n$; this is the property most commonly identified with \ref{cond:rotational}. We will generalize this with assumption \ref{hyp:A5}.
	\item \emph{A limitation to the local \lq\lq unwinding\rq\rq\ of solutions}, namely, the function $t\to \rot_t(z_0)$ can be locally decreasing, but there should be some form of lower bound on $\rot_t(z_0)-\rot_s(z_0)$ for $s<t$. For the superlinear equation \eqref{eq:superlinearODE}, we know that the $\dot x$-axis can be crossed only in the positive sense of rotation, which implies $\rot_t(z_0)-\rot_s(z_0)>-\frac{1}{2}$ for every $s<t$ (cf.~\cite[Prop.~7]{Gid23}). The Ambrosetti--Rabinowitz condition instead implies that $t\mapsto \rot_t(z_0)$ is increasing whenever $\abs{z(t)}>R$ (cf.~Lemma~\ref{lem:H2impliesA6}). We will require only the much weaker condition \ref{hyp:A6}, which is a (possibly negative) uniform lower bound on the angular velocity as $\abs{z}\to+\infty$.
 \end{itemize}

Our main focus will  concern the effects of the global existence of solutions \ref{cond:global}, or the lack thereof. First of all, let us notice that it is not assured by the classical rotational conditions. For the superlinear ODE \eqref{eq:superlinearODE}--\eqref{eq:superlinearODEgrowth}, an example with a solution blowing up in finite time has been provided in \cite{CofUll}. For the superquadratic Hamiltonian system \eqref{eq:HS}--\eqref{eq:Ambrosetti-Rabinowitz},  blow-up solutions are much easier to build: we will discuss a family of examples in Section~\ref{sec:example}.

The advantage of having \ref{cond:global} is clear: it is equivalent to assuming that the solution of the initial value problem \eqref{eq:cauchyIVP} is continuable on $[0,T]$ for every $z_0\in \R^2$, meaning that the Poincaré map $z_0\mapsto\phi(T,z_0)$ associated with the period $T$ is well-defined on the whole plane. This facilitates the application of a suitable fixed-point theorem to recover the existence of periodic solutions, as done for instance in \cite{Fucik,Stru}.

Our first motivating question is thus the following: \emph{if global forward existence of solutions  does not hold, does the problem still admit a periodic solution? Does dropping \ref{cond:global} have qualitative consequences on the rotational properties of the solutions?}

This issue was first addressed by one of the authors for the case of the superlinear ODE \eqref{eq:superlinearODE} in \cite{Gid23}, allowing also a non-Hamiltonian perturbation and introducing a novel rotation-based fixed-point approach. The technique was then employed in \cite{WCQ24}; we also mention \cite{FelFonSfe24} for an alternative recent approach to  a special case of \eqref{eq:superlinearODE}.

A key point in \cite{Gid23} was the assumption of the continuability on $[0,T]$ of the solutions crossing the origin of the phase plane during that time interval, which we will assume in \ref{hyp:A3}. 
Under mild regularity conditions, corresponding to our \ref{hyp:A1}--\ref{hyp:A2}, in \cite{Gid23} it was shown that \ref{hyp:A3} is sufficient and, up to a suitable change of coordinates, also necessary for the existence of a $T$-periodic solution for \eqref{eq:superlinearODE}--\eqref{eq:superlinearODEgrowth}.

The aim of this paper is to extend this result to general planar systems \eqref{eq:planar}, assuming  \ref{hyp:A3}, with \ref{hyp:A5}--\ref{hyp:A6} standing as an \emph{umbrella} rotational condition \ref{cond:rotational}. We therefore ask ourselves: \emph{under this framework, does the special structure of \eqref{eq:superlinearODE} have any nontrivial implication for the rotational properties of the system, that are lost in \eqref{eq:planar}?} 

We have already noticed a small difference when we discussed the ``unwinding'' of solutions in \ref{cond:rotational}. However, the most relevant change concerns the rotational properties of blow-up solutions. In  \eqref{eq:superlinearODE}--\eqref{eq:superlinearODEgrowth}, if a solution blows up at a finite time $\bar t$ without crossing the origin, then $\rot_{t}(x_0,v_0)\to +\infty$ as $t\to\bar t$, namely the solution must do infinite rotations in order to blow up, cf.~\cite[Prop.~7]{Gid23}. As we show in Section~\ref{sec:example}, this is false already for Hamiltonian systems \eqref{eq:HS} satisfying the Ambrosetti--Rabinowitz condition \eqref{eq:Ambrosetti-Rabinowitz}. 

In this work we restrict our analysis to the case where blow-up solutions make infinite rotations, which we will assume with \ref{hyp:A4}. Under such conditions, we prove the existence of a $T$-periodic solution for \eqref{eq:planar}. We will also provide a sufficient condition for \ref{hyp:A4} in the case of Hamiltonian systems, which we then apply to our example in Section \ref{sec:example}. 
The case of systems including blow-up solutions with a finite number of rotations still remains an open problem. Intuitively, we consider this case more problematic, since, in a way, a condition such as \eqref{eq:Ambrosetti-Rabinowitz} is no longer perfectly playing the role of rotational condition \ref{cond:rotational}, because blow-up solutions do not spend enough time at large radii to rotate as desired.

We finally observe that, for Hamiltonian systems \eqref{eq:HS}, the existence of a first $T$-periodic solution is the gateway to proving the existence of infinitely many ones, rotating around the first one, via the Poincaré--Birkhoff Theorem. For the ODE \eqref{eq:superlinearODE} this was first noted by Hartman \cite{Hart} and Jacobowitz \cite{Jacob}, we refer instead to \cite{Bos12,FonSabZan} for applications to planar Hamiltonian systems. Notice that for \eqref{eq:superlinearODE} it has been shown in \cite{FonSfe} that this procedure can be applied also without requiring global existence; we do not exclude that the same procedure could be adapted also to \eqref{eq:HS}, although it is not trivial and beyond the scope of the present work. We also remark once more that we do not require the Hamiltonian structure when stating our main result for \eqref{eq:planar}, although it is the most relevant application.

The paper is structured as follows. In Section~\ref{sec:prel} we recall some preliminary notions on rotation and fixed-point theorems. Our assumptions \ref{hyp:A1}--\ref{hyp:A6} and the main result for \eqref{eq:planar} are presented in Section~\ref{sec:main} and then proved in Section~\ref{sec:proof}. Then, in Section~\ref{section:ham-case}, we focus on the special case of Hamiltonian systems \eqref{eq:HS}, providing some sufficient conditions on the Hamiltonian $H$ in order for the abstract assumption to hold. Finally, in Section~\ref{sec:example} we discuss an illustrative example of a Hamiltonian system with blow-up solutions.

\section{Notation and preliminary results} \label{sec:prel}

 Given a set $U$, $\cl U$ denotes its closure and $\partial U$ its boundary. Open intervals are denoted by $(a,b)$, closed ones by $[a,b]$, with the mixed case $(a,b]$ defined accordingly.

 $\Id_{\R^2}$ is the identity map on $\R^2$ and $0_{\R^2}$ denotes the origin $(0,0)$; whenever there is no ambiguity, the subscript will be omitted. The open ball centered at $z\in\R^{2}$ with radius $R$ is denoted by $B(z,R)\subset \R^{2}$, while the closed ball by $B[z,R]=\cl B(z,R)$. $\abs{v}$ is the Euclidean norm of a vector $v\in \R^n$. 
The overdot $\dot{\ }$ denotes the time derivative $\frac{\dd}{\dd t}$.

\subsection{Rotation and a fixed point theorem} \label{subsec:topol}

To provide a general definition of rotation number, we introduce the \emph{clockwise} polar coordinates $(\theta,r)$, according to the change of coordinates $\Psi\colon\cl\halfPlane \to \R^2$,
\begin{equation}
\label{eq:polar_coords} 
(x,y)=\Psi(\theta,r)=(r\cos\theta,-r\sin\theta),
\end{equation}
where $\halfPlane=\R\times(0,+\infty)$ and $\cl\halfPlane=\R\times[0,+\infty)$ denote the open and closed half-plane, respectively.

Let us consider the continuous evolution map $\phi\colon[0,T]\times E\to \R^{2}$ defined on an open set $E\subseteq\R^2$, with $\phi(0,\cdot)=\Id_{E}$. 
We define the set $\Null$ as
\begin{equation}
    \label{eq:general_set_N}
    \Null=\left\{
    z_{0}\in E:\,\exists t\in[0,T]\text{ such that }\phi(t,z_{0})=0_{\R^{2}}
    \right\}\,,
\end{equation}
namely, the set of initial data $z_{0}$ whose trajectory reaches the origin within $[0,T]$. 

Thus the set $E_{0}\coloneqq E\setminus \Null$ consists of those points in $E$ whose $\phi$-orbit does not pass through the origin of $\R^2$. Recalling that $(\Psi,\halfPlane)$ is a covering space of $\R^2\setminus\{0\}$, and setting $\mathcal{E}_{0}\coloneqq \Psi^{-1}(E_{0})$, \cite[Theorem~11]{Gid23} ensures the existence of a lift $\Phi\colon[0,T]\times \mathcal{E}_{0}\to\halfPlane$ in clockwise polar coordinates of $\phi$, restricted to $E_{0}$, i.e. a continuous $\Phi$ such that
\begin{equation}
    \Phi(0,\cdot)=\Id_{\mathcal{E}_{0}}\,,\qquad \Psi(\Phi(t,\cdot))=\phi(t,\Psi(\cdot))\,.
\end{equation}
We denote the two components of $\Phi$ as $\Phi=(\Phi^\theta,\Phi^r)$.

For any $t\in[0,T]$, the rotation associated with the point $z_0\in E_0$ over the interval $[0,t]$ is defined as
\begin{equation}
\label{eq:abstract_rotation_index}
\rot_t(z_{0})\coloneqq\frac{\Phi^\theta(t,\eta_0)-\Phi^\theta(0,\eta_0)}{2\pi}\,,
\end{equation}
for any $\eta_0=(\theta_0,r_0)\in\Psi^{-1}(z_0)$. The definition is well-posed thanks to the periodicity in $\theta$ of $\Psi$ in \eqref{eq:polar_coords}. Accordingly, for any subset $U\subseteq E_0$ we write $\rot_t(U)\coloneqq\{\rot_{t}(z):z\in U\}$.

Our main result relies on the following fixed-point theorem.
\begin{theorem}
    \label{thm:fixed_point}
    Let $U\subset\R^{2}$ be an open bounded set with $0_{\R^{2}}\in U$. Suppose that $\phi\colon[0,T]\times \cl U\to\R^{2}$ is a continuous map such that $\phi(0,z)=z$ for every $z\in \cl U$ and $\Null\cap \partial U=\emptyset$. If
    \begin{equation}
        \rot_T(\partial U)\cap\Z=\emptyset\,,
    \end{equation}
    then $\deg(\phi(T,\cdot)-\Id_{\R^{2}},U,0)=1$. In particular, $\phi(T,\cdot)$ admits a fixed point $z\in U$.
\end{theorem}
This theorem can be seen as a corollary of the Poincaré--Bohl Theorem, as proved in~\cite[Theorem~1]{Gid23}.

\subsection{The planar system}

In this paper, the evolution map $\phi$ considered will be the one associated with the solution of the planar system of differential equations \eqref{eq:planar}, where $z=(x,y)\in\R^{2}$ and $F\colon\R\times \R^{2}\to \R^{2}$ is $T$-periodic in $t$ and satisfies the Carathéodory conditions. We recall that a function $F(t,z)$, defined on an open subset of $\R\times\R^{d}$, is said to satisfy the Carathéodory conditions if:
    \begin{itemize}
        \item for every $z\in\R^{d}$, the map $t\mapsto F(t,z)$ is measurable;
        \item for almost every $t\in \R$, the map $z\mapsto F(t,z)$ is continuous;
        \item for every compact set $[a,b]\times K\subset \R\times \R^{d}$ there exists a function $m_{K}\in L^{1}([a,b])$ such that $\abs{F(t,z)}\leq m_{K}(t)$ for a.e. $t\in [a,b]$ and every $z\in K$. 
    \end{itemize}
Under such assumptions on $F$, there is local existence of solutions for \eqref{eq:planar} and we have the following continuous dependence property (see, e.g., \cite[Ch.~1, \textsection~1]{Filippov}). 

\begin{theorem}[Continuous dependence]
\label{th:continuous_dependence}
    Let $F = F(t, z) \colon \R \times \R^2 \to \R^2 $ satisfy the Carathéodory conditions. Assume that the solution $z(t)=\zeta(t; \bar t, \bar z)$ of the problem
\begin{equation}\label{eq:Cauchy_gen}
\begin{cases}
    \dot{z}(t) = F(t, z(t)) \\
    z(\bar t) = \bar z
\end{cases}
\end{equation}
exists on the interval $[t_a, t_b]$  and is unique. Then there exists $\delta > 0$ such that, for every $(\hat{t}, \hat{z})$ with
\begin{equation}
\abs{\bar t - \hat{t}} + \abs{\bar z - \hat{z}} < \delta \,,
\end{equation}
the solution $ \zeta(t; \hat{t}, \hat{z}) $ of the Cauchy problem defined by $ z(\hat{t}) = \hat{z} $ exists on the interval $[t_a, t_b]$. Moreover, $ \zeta(s; \hat{t}, \hat{z}) $ is a continuous function of $ (s; \hat{t}, \hat{z}) $ at $(t; \bar t, \bar z) $.
\end{theorem}

Throughout the paper,  we use $z_{0}$ for the initial data of a solution of \eqref{eq:planar} such that $z(0)=z_{0}$, while $(\bar{t},\bar{z})$ for a generic Cauchy problem $z(\bar{t})=\bar{z}$ with $\bar{t}\in[0,T]$, unless otherwise specified.
The function $\zeta(t; \bar{t}, \bar{z})$ denotes the solution of the general Cauchy problem \eqref{eq:Cauchy_gen}, for any value of $t\in\R$ on which it is defined.

For our purposes, we will focus mostly, but not exclusively, on the initial value problem \eqref{eq:cauchyIVP}. In particular, such an initial value problem will be used to introduce the notions of Subsection~\ref{subsec:topol} for the planar system \eqref{eq:planar}. More precisely, we set
\begin{equation}\label{eq:def_phi}
	\phi(t,z_0)\coloneqq \zeta(t;0,z_0)\,,
\end{equation}
where a suitable domain for $\phi$ will be found in the next section.
It is easily verified that the clockwise rotation number of $\zeta(t;0,z_{0})$ introduced in \eqref{eq:abstract_rotation_index} can then be formulated in this framework as
\begin{equation}
	\label{eq:rot_index}
	\rot_t(z_{0})= \frac{1}{2\pi}\int_{0}^{t}\frac{\pscl{F(s,\phi(s,z_0))}{J\phi(s,z_0)}}{\abs{\phi(s,z_0)}^{2}}\,\dd s \,,
\end{equation}
where $J$ is the standard symplectic matrix
\begin{equation*}
	J=\begin{pmatrix}
		0 & 1\\
		-1& 0
	\end{pmatrix}
\end{equation*}
which is also corresponding to a $\frac{\pi}{2}$ clockwise rotation.

It will be useful to extend a generalized notion of rotation also to solutions blowing up in $[0,T]$, but without crossing the origin $0_{\R^2}$ first. More precisely, we define $\GRot\colon [0,T]\times \bigl(\R^{2}\setminus \Null\bigr)\to\R\cup\{+\infty\}$ as
\begin{equation}
	\label{eq:extended_rotat_index}
	\GRot_t(z_{0})=\begin{cases}
		\rot_t(z_{0}) & \text{if it is well defined},\\
		+\infty & \text{otherwise}.
	\end{cases}
\end{equation}
The reason for this choice will be clear in light of Assumption~\ref{hyp:A4}, which will be pivotal to obtain the continuity of $\GRot$ in Lemma~\ref{lem:continuity_GRot}. Note, however, that this definition does not create problems to the application of Theorem~\ref{thm:fixed_point}, since it keeps blow-up solutions apart.

\section{Main result} \label{sec:main}

We study the planar system \eqref{eq:planar} under the following assumptions:
\begin{enumerate}[label=\textup{(A\arabic*)}]
    \item \label{hyp:A1} the function $F$ satisfies the \emph{Carathéodory conditions} and is $T$-periodic in the time variable $t$;
    \item \label{hyp:A2} for every $(\bar{t},\bar{z})\in\R\times \R^{2}$ the solution of the Cauchy problem \eqref{eq:Cauchy_gen} is unique; 
    \item \label{hyp:A3} (continuability hypothesis) for every $\bar{t}\in[0,T]$, the solution to the Cauchy problem 
    \begin{equation*}
    \begin{cases}
        \dot{z}=F(t,z)\\
        z(\bar{t})=0_{\R^{2}}
    \end{cases}
    \end{equation*}
    can be continued to the whole interval $[0,T]$;
    \item \label{hyp:A4} (rotation of blow-up solutions) for every $z_0$ such that the solution of the initial value problem \eqref{eq:cauchyIVP} blows up at a finite time $T_\mathrm{F}\in(0,T]$, we have
    \begin{equation*}
    \lim_{t\to T_\mathrm{F}^{-}}\rot_t(z_0)=+\infty\, ;
    \end{equation*}
    \item \label{hyp:A5} (growth at infinity of $\GRot_T$)  
    \begin{equation*}
    \lim_{\abs{z_0}\to+\infty}\GRot_T(z_0)=+\infty\, ;
    \end{equation*}
    \item \label{hyp:A6} (uniform asymptotic lower bound on the angular velocity) the field $F$ satisfies
    \begin{equation*}
    \liminf_{\abs{z}\to+\infty}\frac{\pscl{F(t,z)}{Jz}}{\abs{z}^{2}}=\ell>-\infty
    \end{equation*}
    uniformly for a.e. $t\in [0,T]$.
\end{enumerate}

Under these six assumptions we obtain the main result of the paper.

\begin{theorem}
    \label{thm:main_theorem}
    Suppose that the six assumptions \ref{hyp:A1}--\ref{hyp:A6} hold. Then the planar system \eqref{eq:planar} admits at least one $T$-periodic solution.
\end{theorem}

We emphasize that \ref{hyp:A3} is not a global condition, but can be achieved by a local bound on a suitable neighbourhood of the origin. For instance, \ref{hyp:A3} is satisfied if 
\begin{enumerate}[label=\textup{($\star$)}]
\item there exists $c>0$ such that $\abs{\langle F(t,z),z\rangle}\leq c \abs{z}$ for a.e.~$t\in [0,T]$ and every $z\in B(0,cT)$. \label{cond:A3estimate}
\end{enumerate}
More sophisticated conditions can be proposed in the same spirit.

%%%%%%%%%%%%%%%%%%%%%%%%%%%%%%%%%%%%%%%%%%%%%%%%%%%%%%%%%%

\section{Proof of Theorem \ref{thm:main_theorem}} \label{sec:proof}

\subsection{Behaviour in a neighbourhood of $\Null$}
Let \ref{hyp:A1} and \ref{hyp:A2} hold. We define $\zeta_0$ as the function that associates to a pair $(\bar t,\bar z)$ the value $\zeta_0\coloneqq \zeta(0;\bar t,\bar z)$ at time $t=0$ of the solution $\zeta(t;\bar t,\bar z)$ of the Cauchy problem \eqref{eq:Cauchy_gen}, if such solution can be continued to $t=0$. Clearly the domain $\dom \zeta_0$ might be a strict subset of $\R\times\R^2$. However, by \ref{hyp:A3}, we know that $[0,T]\times \{0_{\R^2}\}\subset\dom \zeta_0$ and therefore, by \eqref{eq:general_set_N}, that $\Null=\zeta_0([0,T]\times \{0_{\R^2}\})$. 
We plan to extend this structure to sufficiently small cylinders near $[0,T]\times \{0_{\R^2}\}$.

\begin{prop}\label{prop:Dcompact}	
	Let \ref{hyp:A1}, \ref{hyp:A2} and \ref{hyp:A3} hold.
	Then there exists $r^*>0$ such that for every $(\bar t,\bar z)\in[0,T]\times B(0,r^*)$ the solution of the corresponding Cauchy problem \eqref{eq:Cauchy_gen} can be continued on the interval $t\in[0,T]$. In particular, this implies that $[0,T]\times B(0,r^*)\subseteq \dom \zeta_0$. Moreover, defining for $0\leq \delta<r^*$ the set
    \begin{equation}
		\label{eq:def_D_delta}
		\DD_{\delta}\coloneqq\{z:\;\exists t\in[0,T]\ \text{such that}\ \zeta(t;0,z)\in B[0,\delta]\}\,,
	\end{equation}
all such sets $\DD_\delta$ are compact.
\end{prop}
\begin{proof}
Let us take any $t_\alpha\in[0,T]$. By \ref{hyp:A3} and the fact that the maximal interval of existence is always open, there exists $s_\alpha>0$ such that the solution $\zeta(t;t_\alpha,0)$ can be continued on the closed interval $[0-s_\alpha,T+s_\alpha]$. By Theorem~\ref{th:continuous_dependence} there exists a cylindrical open neighbourhood $(t_\alpha^-,t_\alpha^+)\times B(0,r_\alpha)$ of $(t_\alpha,0_{\R^2})$ such that for every $(\bar t,\bar z)\in (t_\alpha^-,t_\alpha^+)\times B(0,r_\alpha)$ the solution $\zeta(t;\bar t,\bar z)$ of the Cauchy problem \eqref{eq:Cauchy_gen} can be continued to $[0-s_\alpha,T+s_\alpha]$.

Since all the open sets $(t_\alpha^-,t_\alpha^+)\times B(0,r_\alpha)$ form an open cover of the compact set $[0,T]\times \{0_{\R^2}\}$, there exists a finite subcover with indices $\alpha_1,\dots,\alpha_\ell$. Let us set
\begin{align*}
t_*^-\coloneqq\min_{i=1,\dots,\ell} t^-_{\alpha_i}<0 \,, && t_*^+\coloneqq\max_{i=1,\dots,\ell} t^+_{\alpha_i}>T \,, \\ r^*\coloneqq\min_{i=1,\dots,\ell} r_{\alpha_i}>0 \,,&& s^*\coloneqq\min_{i=1,\dots,\ell} s_{\alpha_i}>0 \,.
\end{align*}
By construction, the set $(t_*^-,t_*^+)\times B(0,r^*)$ is an open neighbourhood of $[0,T]\times\{0_{\R^2}\}$ and, for every $(\bar t,\bar z)\in[0,T]\times B(0,r^*)$ the solution of the Cauchy problem \eqref{eq:Cauchy_gen} can be continued on $[0-s^*,T+s^*]$. 

This means that the function  $\zeta_0$ is well-defined and continuous on the open set $(t_*^-,t_*^+)\times B(0,r^*)$, of which the sets $[0,T]\times B[0,\delta]$ are compact subsets for every $0\leq\delta<r^*$. Since, by construction, $\DD_\delta=\zeta_0([0,T]\times B[0,\delta])$, it follows that the sets $\DD_\delta$ are compact being the images of compact sets with respect to the continuous map $\zeta_0$.
\end{proof}	

Since $\DD_0=\Null$, we have proved the following fact. 

\begin{corollary}\label{cor:Ncompact}
	If \ref{hyp:A1}, \ref{hyp:A2} and \ref{hyp:A3} hold, then $\Null$ is compact.
\end{corollary}

This fact is essential for a rotational approach, since $\Null$ is the set of initial points for which no notion of rotation can be properly defined. 

Proposition~\ref{prop:Dcompact} has a second key consequence, namely that the rotation is finite on $\DD_\delta\setminus\Null$, since there is no blow-up solution in $[0,T]$ for such initial points. In order to fully exploit this property, we need the following lemma.

\begin{lemma}
	\label{lem:Ddelta_chain}
	Suppose that \ref{hyp:A1}, \ref{hyp:A2}, and \ref{hyp:A3} hold. Then, for every $\delta_1,\delta_2$ satisfying $0\leq \delta_1<\delta_2<r^*$,  we have  $\DD_{\delta_1}\subseteq \interior\DD_{\delta_2}$.
\end{lemma}

\begin{proof}It is sufficient to prove that for every $z_0\in \DD_{\delta_1}$ there exists a radius $r_0>0$ such that $B(z_0,r_0)\subseteq \DD_{\delta_2}$. 
	
For every $z_0\in \DD_{\delta_1}$, there exists $\bar t\in[0,T]$ such that $\zeta(\bar t;0,z_0)=\bar z\in B[0,\delta_1]$. By continuous dependence on initial data, there exists $r_0 > 0$ such that, for every $z\in \R^2$ with $\abs{z - z_0} < r_0$, we have
\begin{equation*}
	\abs{\zeta(\bar{t};0, z) - \zeta(\bar{t};0, z_0)} < \delta_2-\delta_1 \,.
\end{equation*}
Hence,
\begin{equation*}
	\abs{\zeta(\bar{t};0, z)} \leq \abs{\zeta(\bar{t};0, z_0)} + \abs{\zeta(\bar{t};0, z) - \zeta(\bar{t};0, z_0)}
	< \delta_1 + (\delta_2-\delta_1) =\delta_2 \,.
\end{equation*}
Therefore $B(z_0, r_0)\subseteq \DD_{\delta_2}$, concluding the proof.
\end{proof}

\subsection{Rotational properties}

\begin{lemma}[Lower estimate on rotation]
    \label{lem:rotation_lowerbound}
    Assume \ref{hyp:A1}, \ref{hyp:A2}, \ref{hyp:A3}, \ref{hyp:A6} hold, and take any $\delta$ with $0<\delta<r^*$. There exists an integrable non-negative function $a_\delta\colon[0,T]\to \R$ such that for every $z_{0}\in \R^{2}\setminus \DD_{\delta}$, if the solution $\zeta(t;0,z_{0})$ of \eqref{eq:cauchyIVP} exists on $[0,T]$, then
     \begin{equation}
    	\label{eq:rotation_lowerbound}
    	\rot_t(z_{0})\geq\rot_s(z_{0})-\int_s^t a_\delta(\tau)\dd \tau\qquad \text{for every  $0\leq s\leq t\leq T$.}
    \end{equation}
\end{lemma}
\begin{proof}
    By \ref{hyp:A6}, there exists $R>\delta$ such that 
    \begin{equation*}
    \frac{\pscl{F(t,z)}{Jz}}{\abs{z}^{2}}\geq \ell-1\qquad \text{for a.e. $t\in[0,T]$ and every $\abs{z}\geq R$.}
    \end{equation*}
  Moreover, the set
    $[0,T]\times \bigl( B[0,R]\setminus B(0,\delta) \bigr)$
is compact. Hence, by \ref{hyp:A1} there exists a non-negative integrable function $m\in L^1([0,T])$ such that
\begin{equation*}
\frac{\pscl{F(t,z)}{Jz}}{\abs{z}^{2}}\geq-m(t) \qquad \text{for a.e. $t\in[0,T]$ and every $z\in B[0,R]\setminus B(0,\delta)$.}
\end{equation*}
 We set $a_\delta(t)\coloneqq \frac{1}{2\pi} \max\{1-\ell,m(t)\}\geq0$  and recall that, since $z_{0}\in \R^{2}\setminus \DD_{\delta}$, we have $\abs{\zeta(t;0,z_{0})}>\delta$ for every $t\in[0,T]$. This gives the estimate
    \begin{equation}\label{eq:local_rotation_estimate}
	\frac{\pscl{F(t,\zeta(t;0,z_{0}))}{J\zeta(t;0,z_{0})}}{\abs{\zeta(t;0,z_{0})}^{2}} \geq - 2\pi a_\delta(t)\qquad \text{for a.e. $t\in[0,T]$.}
\end{equation}
In light of \eqref{eq:rot_index}, integrating the estimate \eqref{eq:local_rotation_estimate} over the interval $[s,t]$ yields \eqref{eq:rotation_lowerbound}.
\end{proof}

Thanks to this lemma, we can prove the continuity of the generalized rotation $\GRot_t(z_{0})$, justifying its definition given in \eqref{eq:extended_rotat_index}.

\begin{lemma}
    \label{lem:continuity_GRot}
     Suppose that the five assumptions \ref{hyp:A1}, \ref{hyp:A2}, \ref{hyp:A3}, \ref{hyp:A4} and \ref{hyp:A6} hold. Then, for every $0<\delta<r^*$, the function $\GRot_t(z_0)$ is continuous with respect to $(t,z_0)$  on $[0,T]\times\bigl(\R^{2}\setminus\DD_{\delta}\bigr)$.
\end{lemma}
\begin{proof}
	Take any $(\hat t, z_0)\in [0,T]\times\bigl(\R^{2}\setminus\DD_{\delta}\bigr)$. We distinguish between two cases. 
	
			If the solution of the initial value problem \eqref{eq:cauchyIVP} exists on $[0,\hat t]$, then $\zeta(\hat t;0,z_0)$ is well-defined and finite. Thus, by the definition of $\DD_\delta$ and Theorem~\ref{th:continuous_dependence}, the rotation $\GRot_{\hat t}(z_{0})$ is well-defined, finite and continuous at $(\hat t,z_0)$.
	
Let us now instead assume that the solution of the initial value problem \eqref{eq:cauchyIVP} blows up at a finite time $T_\mathrm{F}\in(0,T]$ and consider $\hat t\in[T_\mathrm{F},T]$. By \ref{hyp:A4} we have $\GRot_{\hat t}(z_{0})=+\infty$. Therefore, we must show that, for every $k\in\N$, there exist a neighbourhood $U_{k}\subseteq\R^2\setminus\DD_\delta$ of $z_0$ and a time $t_k\in[0,\hat t)$ such that 
    \begin{equation}
        \GRot_t(z)>k \qquad \text{for every $(t,z)\in [t_{k},T]\times U_{k}$\,.}
    \end{equation}
    Let $\Lambda\coloneqq \int_0^T a_\delta(t)\dd t$, with $a_\delta$ from Lemma~\ref{lem:rotation_lowerbound}. By \ref{hyp:A4}, for every $k>0$ there exists $t_{k}<T_\mathrm{F}$ such that $\GRot_{t_{k}}(z_{0})>k+\Lambda$.  Since $t_{k}<T_\mathrm{F}$, $\GRot_{t_{k}}(z_{0})<+\infty$ and, by the previous case, there exists a neighbourhood $U_{k}\subseteq\R^2\setminus\DD_\delta$ of $z_{0}$ such that 
    \begin{equation}
    	\label{eq:blowup_rotation_estimate}
        \GRot_{t_{k}}(z)>k+\Lambda \qquad \text{for every $z\in U_{k}$}\,.
    \end{equation}
    Finally, by Lemma~\ref{lem:rotation_lowerbound} and \eqref{eq:blowup_rotation_estimate},  we obtain
    \begin{equation*}
    \GRot_{t}(z)\geq \GRot_{t_{k}}(z)-\int^t_{t_{k}}a_\delta(s)\dd s>k \qquad \text{for every $t\in(t_k,T], z\in U_{k}$\,,}
    \end{equation*}
    which concludes the proof.
\end{proof}
Lemma~\ref{lem:continuity_GRot} implies straightforwardly that $\GRot_t(z_0)$ is continuous with respect to $(t,z_0)$  on its whole domain $[0,T]\times\bigl(\R^{2}\setminus\Null\bigr)$.

\subsection{Proof of Theorem \ref{thm:main_theorem}} \label{subsec:mainproof}
    
   First of all, observe that a  $T$-periodic solution for \eqref{eq:planar} corresponds to a fixed point of the map $\phi(T,\cdot)$ introduced in \eqref{eq:def_phi}. To find such a fixed point, we plan to apply Theorem~\ref{thm:fixed_point} on a suitably constructed set $U$.

   Let us take any $\delta\in(0,r^*)$. By Proposition~\ref{prop:Dcompact}, $\phi(T,\cdot)$ is well-defined on $\DD_\delta$ and by Lemma~\ref{lem:Ddelta_chain} we have
    \begin{equation*}
   	\{0_{\R^2}\}\subseteq \Null\subset \interior \DD_\delta \subset \DD_\delta \,.
   \end{equation*}
This implies that the rotation number $\rot_T(z_0)$ is well defined and finite for all $z_0\in \partial \DD_\delta$.
Thanks to the compactness of $\partial \DD_\delta$ (Proposition~\ref{prop:Dcompact}) and the continuity of $\rot_T(z_0)$ on that set (Lemma~\ref{lem:continuity_GRot}), there exists an integer $\bar{n}\in\N$ such that 
\begin{equation}
	\label{eq:bound_from_above_rot_on_U1}
	\rot_T(z_0)< \bar{n} \qquad \text{for every $z_0\in\partial\DD_\delta$\,.}
\end{equation}

Consider the set 
\begin{equation*}
	W=\left\{z_0\in\R^{2} \setminus\Null:\;\GRot_T(z_0)<\bar{n}+\tfrac{1}{2}\right\}.
\end{equation*}
The set $W$ is open (in $\R^2\setminus\Null$, hence also in $\R^2$) since it is the preimage of an open set. Moreover, by assumption \ref{hyp:A5},  $W$ is bounded. Therefore, the set
\begin{equation}
	\label{eq:new_u_for_main_thm}
	U\coloneqq\interior \DD_\delta \cup W
\end{equation}
is open and bounded, being the union of two open bounded sets. Note that $\Null\cap\partial U=\emptyset$, since $\Null\subset \interior \DD_\delta$ and $\interior \DD_\delta$ is open. By the continuity of $\GRot_T(z_0)$, since by construction $\partial \DD_{\delta}\subset \interior W$, it follows that 
\begin{equation*}
	\GRot_T(\partial U)=\left\{\bar{n}+\tfrac{1}{2}\right\}.
\end{equation*}
In particular, $\GRot_{T}(\partial U)\cap\Z=\emptyset$. Hence, all the assumptions of Theorem~\ref{thm:fixed_point} are satisfied, which gives the existence of a fixed point for $\phi(T,\cdot)$ in $U$, proving Theorem~\ref{thm:main_theorem}.\hfill\qed

\section{The Hamiltonian case}
\label{section:ham-case}

We now discuss our main result in the special case of planar Hamiltonian systems \eqref{eq:HS}. As is customary, to ensure adequate rotational properties at infinity, we will assume the Ambrosetti--Rabinowitz condition \eqref{eq:Ambrosetti-Rabinowitz}. If we were assuming global existence of solutions, so that $\rot_T$ would be well-defined everywhere on $
\R^2\setminus\Null$, the Ambrosetti--Rabinowitz condition \eqref{eq:Ambrosetti-Rabinowitz} would be sufficient for $\rot_T(z_0)$ to be continuous and tend to $+\infty$ as $\abs{z_0}\to+\infty$, leading easily to the existence of a $T$-periodic solution, since the argument of Section \ref{subsec:mainproof} could be directly applied. 
As we will show, this is no longer generally true if global existence of solutions is dropped.

To proceed, let us first set our regularity assumptions on \eqref{eq:HS}:
\begin{enumerate}[label=\textup{(H0)}]
	\item The Hamiltonian function $H(t,z)$ is continuously differentiable in $z$ and \\ \mbox{$F(t,z)\coloneqq J\nabla_{z}H(t,z)$} satisfies conditions \ref{hyp:A1}, \ref{hyp:A2} and \ref{hyp:A3}. 
	\label{hyp:H0}
	\end{enumerate}
Under such a condition, we will be able to focus on what is needed for the rotational properties \ref{hyp:A4}, \ref{hyp:A5} and \ref{hyp:A6} to hold.

To make our exposition easier to follow, let us formulate our discussion in the clockwise \emph{canonical polar coordinates} $(\theta,\rho)$.  
Precisely, given the Hamiltonian function $H(x,y)$, we set
\begin{equation}
	\label{eq:canonical-polar-coordinates}
	(x,y)=(\sqrt{2\rho}\cos\theta,-\sqrt{2\rho}\sin\theta)\,,
\end{equation}
with \emph{clockwise} orientation.

In these coordinates, the Hamiltonian system takes the form
\begin{equation}
	\label{eq:ham-system-canonical-polar-coordinates}
	\begin{cases}		
    \dot{\theta}=\dfrac{\partial H(t,\theta,\rho)}{\partial \rho}\,,\\[10pt]
	\dot{\rho}=-\dfrac{\partial H(t,\theta,\rho)}{\partial \theta}\,.
	\end{cases}
\end{equation}
In the following, we omit the explicit dependence of $x,y$ on $(\theta,\rho)$. Notice in particular that $\abs{z}=\sqrt{2\rho}$.
Within this framework,  we consider the following assumptions.
\begin{enumerate}[label=\textup{(H\arabic*)}]
\item \label{hyp:H1} There exist a large enough radius $\rr >0$ and two constants $c>0$, $\gamma \geq 1$ such that for a.e.~$t\in[0,T]$ and every $\theta\in[0,2\pi)$, $\rho\geq \rr$ it holds
\begin{equation*}
    \frac{\partial H(t,\theta,\rho)}{\partial \rho}\geq c\abs{\frac{\partial H(t,\theta,\rho)}{\partial \theta}}^{\gamma}\,.
\end{equation*}

\item \label{hyp:H2} There exist a large enough radius $\rr>0$, and two constants $k\in(0,1/2)$, $m>0$ such that for a.e.~$t\in[0,T]$ and every $\theta\in[0,2\pi)$, $\rho\geq \rr$ it holds
\begin{equation*}
   0< m\leq H(t,\theta,\rho)\leq 2k\rho \,\frac{\partial H(t,\theta,\rho)}{\partial \rho}\,.
\end{equation*}
\end{enumerate}

\begin{remark}
    Assumption \ref{hyp:H1} can be restated more compactly as $\dot{\theta}\geq c\abs{\dot{\rho}}^{\gamma}$ for $\rho>\rr$. Equivalently, in Cartesian coordinates it corresponds to
    \begin{equation*}
    	\frac{\pscl{\dot z}{Jz}}{\abs{z}^{2}}\geq C\abs{\pscl{\dot{z}}{z}}^{\gamma} \qquad \text{for $\abs{z}>R\coloneqq\sqrt{2\rr}$}\,,
    \end{equation*}
    since we recall that $\dot{\theta}=\frac{\pscl{\dot z}{Jz}}{\abs{z}^{2}}$ and $\dot{\rho}=\pscl{\dot{z}}{z}$.
    
    Assumption \ref{hyp:H2} is exactly the Ambrosetti--Rabinowitz condition \eqref{eq:Ambrosetti-Rabinowitz}, rewritten in polar coordinates. Usually, the lower bound appears as $0<H(t,z)$: this is trivially equivalent to our formulation if $H$ is continuous, as is often assumed in literature (while in \ref{hyp:H0} we assume only Carathéodory regularity). Indeed, first applications of this condition to Hamiltonian systems were in an autonomous setting \cite{BaBe,Rab78}. The existence of such an $m$ ensures a superquadratic growth $H(t,z)\geq a  \abs{z}^{1/k}$, for every $\abs{z}>R$, with a coefficient $a$ uniform in time (cf.~Lemma~\ref{lem:growth_H_under_Ambrosetti-Rabinowitz}), so we consider our formulation of \ref{hyp:H2} a natural extension of the classical condition to the Carathéodory setting.
\end{remark}

\begin{theorem}
	\label{thm:main_hamiltonian}
	Suppose that \ref{hyp:H0}, \ref{hyp:H1} and \ref{hyp:H2} hold. Then the planar Hamiltonian system \eqref{eq:HS} admits at least one $T$-periodic solution.
\end{theorem}

\subsection{Proof of Theorem~\ref{thm:main_hamiltonian}}
Our plan is to prove that under such assumptions conditions \ref{hyp:A4}, \ref{hyp:A5} and \ref{hyp:A6} are satisfied, hence Theorem~\ref{thm:main_theorem} can be applied. We do this through several lemmas.

The proof of \ref{hyp:A6} is a straightforward consequence of \ref{hyp:H2}.
\begin{lemma}
	\label{lem:H2impliesA6}
	If \ref{hyp:H2} holds, then \ref{hyp:A6} holds with $\ell\geq 0$.
\end{lemma}
\begin{proof}
	Since  \ref{hyp:H2} is equivalent to \eqref{eq:Ambrosetti-Rabinowitz},  for every $\abs{z}>R\coloneqq\sqrt{2\rr}$ we have
\begin{equation*}
	\pscl{F(t,z)}{Jz}
	=\pscl{J\nabla_{z}H(t,z)}{Jz}
	=\pscl{\nabla_{z}H(t,z)}{z}
	\geq \frac{1}{k}H(t,z)>0\,.
\end{equation*}
\end{proof}
On the other hand, assumption \ref{hyp:H1} is pivotal to obtain \ref{hyp:A4}.
\begin{lemma}
    \label{lem:hs_blowup_rot}
     If \ref{hyp:H0} and \ref{hyp:H1} hold, then \ref{hyp:A4} holds.
\end{lemma}
\begin{proof}Let $z(t)$ be the solution of an initial value problem
	    \begin{equation}\label{eq:HS_IVP}
		\begin{cases}
			\dot{z}(t)=J\nabla_zH(t,z)\\
			z(0)=z_0
		\end{cases}
	\end{equation} 
blowing up at a finite time $T_\mathrm{F}\in(0,T]$. By \ref{hyp:H0}, this solution never crosses the origin during $t\in[0,T_\mathrm{F})$, hence we can properly express it in canonical polar coordinates \eqref{eq:canonical-polar-coordinates} as $(\theta(t),\rho(t))$.
     
     Furthermore, there exists $\bar{t}\in[0,T_\mathrm{F})$ such that $\rho(t)>\rr$ for all $t\in[\bar{t},T_\mathrm{F})$.
     On this interval, assumption \ref{hyp:H1} gives
    \begin{equation*}
    \dot{\theta}(t)=\frac{\partial H}{\partial \rho}\geq c\abs{\dot{\rho}(t)}^{\gamma}\,,\qquad \gamma \geq 1\,.
    \end{equation*}
We recall that, since $u\mapsto u^{\gamma}$ with $\gamma \geq 1$ is convex, by Jensen's inequality we obtain
\begin{equation}
	\label{eq:Jensen-ineq}
	\frac{1}{b-a}\int_{a}^{b}g(s)^{\gamma}\dd s\geq \left(\frac{1}{b-a}\int_{a}^{b}g(s)\dd s\right)^{\gamma} 
\end{equation}
for every non-negative integrable function $g$ on $[a,b]$.
For $t\in(\bar t, T_\mathrm{F})$, it holds
\begin{align}
\frac{1}{2\pi}\int_{\bar{t}}^{t}\dot{\theta}(s)\dd s 
&\stackrel{\ref{hyp:H1}}{\geq} \frac{c}{2\pi}\int_{\bar{t}}^{t}\abs{\dot{\rho}(s)}^{\gamma} \dd s
\stackrel{\eqref{eq:Jensen-ineq}}{\geq} \frac{c}{2\pi\abs{t-\bar t}^{\gamma-1}}\left(\int_{\bar{t}}^{t}\abs{\dot{\rho}(s)} \dd s\right)^{\gamma} \notag\\
&\geq \frac{c}{2\pi\abs{T_\mathrm{F}-\bar t}^{\gamma-1}}\left(\int_{\bar{t}}^{t}\abs{\dot{\rho}(s)} \dd s\right)^{\gamma} \,.\label{eq:blow_up_norm_inequality}
\end{align}
Since the solution blows up, we have
\begin{equation}
	\label{eq:blow-up}
	\lim_{t \to T_\mathrm{F}} \int_{\bar{t}}^{t} \abs{\dot{\rho}(s)}\dd s
	\geq \lim_{t \to T_\mathrm{F}} \int_{\bar{t}}^{t} \dot{\rho}(s)\dd s
	= \lim_{t \to T_\mathrm{F}} \rho(t)-\rho(\bar{t})=+\infty\,.
\end{equation}
Combining \eqref{eq:blow_up_norm_inequality} with \eqref{eq:blow-up}, recalling that $\gamma \geq 1$, we get
    \begin{equation*}
        \lim_{t \to T_\mathrm{F}} \rot_t(z_0)= \frac{\theta(\bar t)-\theta(0)}{2\pi}
        + \lim_{t \to T_\mathrm{F}} \frac{1}{2\pi}\int_{\bar{t}}^{t}\dot{\theta}(s)\dd s =+\infty\,,
    \end{equation*}
proving \ref{hyp:A4}.
\end{proof}

The proof of \ref{hyp:A5} is more complex and requires two preliminary lemmas. 
To start, we recall that the  Ambrosetti--Rabinowitz condition \ref{hyp:H2} can be seen as superquadratic-growth condition for the Hamiltonian $H$.

\begin{lemma}[Superquadratic growth under~\ref{hyp:H2}]
	\label{lem:growth_H_under_Ambrosetti-Rabinowitz} 
	Let \ref{hyp:A1}, \ref{hyp:A2} and \ref{hyp:H2} hold. Then there exists $a>0$ such that
	\begin{equation*}
		H(t,z)\geq a  \abs{z}^{1/k}  \qquad\text{for every $\abs{z}>R$, a.e. $t\in[0,T]$}
	\end{equation*}
	where $k\in\left(0,\frac{1}{2}\right)$ and $R\coloneqq\sqrt{2\rr}>0$ are the constants from \ref{hyp:H2}.
\end{lemma}
\begin{proof}
	Fix any $t\in[0,T]$ and set $z=r\xi$, where $\xi\in\R^{2}$ is a unit vector and $r> 0$. Define $h(t,r)=H(t,r\xi)$. Then
	\begin{equation*}
		\pscl{\nabla_{z}H(t,z)}{z}= r\,\frac{\partial}{\partial r}h(t,r)\,.
	\end{equation*}
	By \ref{hyp:H2}, we obtain
	\begin{equation*}
		0< h(t,r)\leq k\, r\frac{\partial}{\partial r}h(t,r) \qquad \text{for every $r>R$}\,.
	\end{equation*}
	Dividing the equation by $kr h(t,r)>0$ and then taking the definite integral from $R$ to any $ r>R$  yields
		\begin{equation}
        \label{eq:ineq-h(t,r)}
0<\frac{1}{k}(\ln  r -\ln R) \leq \ln h(t, r)-\ln h(t,R)
	\end{equation}
Let us set
\begin{align*}
	\tilde a\coloneqq  \inf_{t\in[0,T]} \min_{\abs{\xi}=1} H(t,R\xi)\geq m 
	&&\text{and}&&
 a\coloneqq\begin{cases}
     \tilde{a},&\text{if }R\leq 1\\
     \tilde{a} R^{-\frac{1}{k}}&\text{if }R>1
 \end{cases}   
\end{align*}
where the bound $\tilde{a}\geq m>0$ is a consequence of the lower bound in \ref{hyp:H2}  and implies that $a>0$. Thus, by equation \eqref{eq:ineq-h(t,r)}, we have that $a>0$ and
\begin{equation*}
	\ln(h(t,r))\geq\frac{1}{k}\ln r+\ln a \,.
\end{equation*}
 Hence
	\begin{equation*}
		h(t,r)\geq a   r^\frac{1}{k}\qquad  \qquad \text{for every $r>R$}
	\end{equation*}
	which, since the value of $a$ does not depend on $t$, proves the claim.
\end{proof}

The next lemma illustrates the rotational properties of solutions far from the origin.

\begin{lemma}
    \label{lem:rotation_large} 
 Let \ref{hyp:A1}, \ref{hyp:A2} and \ref{hyp:H2} hold. Then, for every $n\in\N$ there exists a radius $\tilde{R}_{n}>0$ such that, for every solution $\zeta(t;0,z_0)$ of the initial value problem \eqref{eq:HS_IVP}  satisfying 
 \begin{equation}\label{eq:stay_afar}
    \abs{\zeta(t;0,z_{0})}\geq\tilde{R}_{n}\qquad \text{for every $t\in[0,T]$ at which it is defined,}
\end{equation}
it holds $\GRot_T(z_{0})>n$.
\end{lemma}
\begin{proof}
	By \eqref{eq:stay_afar}, we have $z_0\notin \Null$, hence $\GRot_T(z_0)$ is well-defined. If $\zeta(t;0,z_{0})$ blows up within $(0,T]$, the conclusion is trivial by the definition \eqref{eq:extended_rotat_index} of $\GRot$. 
	
	We now consider the case where $\zeta(t;0,z_{0})$ is defined on $[0,T]$.
    By assumption \ref{hyp:H2} and Lemma~\ref{lem:growth_H_under_Ambrosetti-Rabinowitz}, we obtain that  $H(t,z)\geq a\abs{z}^{1/k}$ for $\abs{z}>R$. By \ref{hyp:H2} we deduce
    \begin{equation*}
    \dot{\theta}=\frac{\partial}{\partial \rho}H(t,\theta,\rho)\geq \frac{1}{2k\rho}H(t,\theta,\rho) 
    \geq\frac{1}{\rho}H(t,\theta,\rho)\geq \frac{a}{\rho}\left(\sqrt{2\rho}\right)^{\frac{1}{k}}= a\, 2^{\frac{1}{2k}} \, \rho^{\frac{1}{2k}-1}\,.
    \end{equation*}
    Since $\tfrac{1}{2k}-1>0$, we deduce that for every $n>0$ there exists $\tilde{R}_{n}$ such that 
    \begin{equation}\label{eq:dottheta_bound}
    \dot{\theta}\bigr\rvert_{(t,z)}>\frac{2\pi n}{T} \qquad \text{for a.e. $t\in[0,T]$ and every $\abs{z}\in [\tilde{R}_n,+\infty)$}\,.
    \end{equation}
Given such $\tilde{R}_{n}$, if $\abs{\zeta(t;0,z_0)}>\tilde{R}_{n}$ for a.e. $t\in[0,T]$, integrating \eqref{eq:dottheta_bound} over $[0,T]$ completes the proof.
\end{proof}

We are now ready to recover \ref{hyp:A5}.

\begin{lemma}
    \label{lem:hs_proofA5}
	Let \ref{hyp:H0}, \ref{hyp:H1} and \ref{hyp:H2} hold. Then also \ref{hyp:A5} holds. In particular, for every $n\in\N$ there exists a radius $R_{n}>0$ such that, for every $\abs{z_{0}}>R_{n}$, we have $\GRot_T(z_{0})>n$.
\end{lemma}

\begin{proof}
	By \ref{hyp:H0} we can apply Proposition~\ref{prop:Dcompact} and take a sufficiently small $\delta>0$, precisely $\delta<r^*$, such that $\DD_{\delta}$ is compact. Then, there exists $R_\delta>0$ such that $\abs{z_0}<R_\delta$ for every $z_0\in \DD_\delta$. In particular, by Lemma~\ref{lem:Ddelta_chain}, $\abs{z_0}<R_\delta$ for every $z_0\in \Null$, so that $\GRot_T(z_0)$ is well-defined for $\abs{z_0}>R_\delta$.
	
Let $\Lambda\coloneqq \int_0^T a_\delta(t)\dd t \geq 0$, with $a_\delta$ from Lemma~\ref{lem:rotation_lowerbound}, 
 and pick any
\begin{equation}\label{eq:defin_Rn}
R_n>\max\left\{ R_\delta,\tilde R_n,\sqrt{\tilde{R}_{n}^{2}+2\left[\frac{2\pi(n+\Lambda)}{T^{(1-\gamma)}c}\right]^{1/\gamma}} \right\} \,,
\end{equation}
where $c$ and $\gamma$ are the constants from \ref{hyp:H1}.

	For any $z_0$ with $\abs{z_0}>R_n$, let $(\theta(t),\rho(t))$ be the description in canonical polar coordinates \eqref{eq:canonical-polar-coordinates} of the solution $\zeta(t;0,z_0)$. We distinguish between three alternative situations:
	\begin{itemize}
		\item if $\zeta(t;0,z_0)$ blows up within $(0,T]$, the claim follows by the definition \eqref{eq:extended_rotat_index} of $\GRot$;
		\item if $\abs{\zeta(t;0,z_0)}\geq \tilde R_n$ for every $t\in [0,T]$ the claim follows from Lemma~\ref{lem:rotation_large};
		\item else, let $\bar t\coloneqq \inf\{t\in(0,T] : \abs{\zeta(t;0,z_0)}<\tilde R_n\}\in(0, T)$ and proceed as follows.
	\end{itemize}
 Recalling that $\abs{z}=\sqrt{2\rho}$, we have $\rho(0)\geq R_{n}^{2}/2$.
Since by continuity of the solution $\abs{\zeta(\bar t;0,z_0)}=\tilde R_n$, using \eqref{eq:defin_Rn} we have
    \begin{equation}
    	\label{eq:estimate_rho}
    \int_{0}^{\bar{t}}\abs{\dot{\rho}(s)}\dd s
    \geq \int_{\bar{t}}^{0}\dot{\rho}(s)\dd s
    =\rho(0)-\rho(\bar{t}) 
    \geq \frac{R_{n}^{2}}{2}-\frac{\tilde{R}_{n}^{2}}{2}>\left[\frac{2\pi(n+\Lambda)}{T^{(1-\gamma)}c}\right]^{1/\gamma}\,.
    \end{equation}
    Thus,
    \begin{align*}
    	\rot_{\bar t}(z_0) &=\frac{1}{2\pi}\int_{0}^{\bar{t}}\dot{\theta}(s)\dd s
    	\stackrel{\ref{hyp:H1}}{\geq}\frac{1}{2\pi}\int_{0}^{\bar{t}}c\abs{\dot{\rho}(s)}^{\gamma}\dd s \\
    	&\stackrel{\eqref{eq:Jensen-ineq}}{\geq}\frac{\bar t^{(1-\gamma)}c}{2\pi}
         \left(\int_{0}^{\bar{t}}\abs{\dot{\rho}(s)}\dd s \right)^\gamma
     \stackrel{\eqref{eq:estimate_rho}}{>} \frac{T^{(1-\gamma)}c}{2\pi} \left[\frac{2\pi(n+\Lambda)}{T^{(1-\gamma)}c}\right]=n+\Lambda \,.
    \end{align*}
This last estimate, combined with Lemma~\ref{lem:rotation_lowerbound}, yields
	\begin{equation}
\rot_T(z_{0})\geq\rot_{\bar t}(z_{0})-\int_{\bar t}^T a_\delta(s)\dd s > n+\Lambda -\int_{\bar t}^T a_\delta(s) \dd s \geq n \,,
\end{equation}
completing the proof also in the third case.
\end{proof}

\begin{proof}[Proof of Theorem~\ref{thm:main_hamiltonian}]
By \ref{hyp:H0}  and Lemmas~\ref{lem:hs_blowup_rot}, \ref{lem:hs_proofA5} and \ref{lem:H2impliesA6}, the six conditions \ref{hyp:A1}--\ref{hyp:A6} hold for  \eqref{eq:HS}, so we can apply Theorem~\ref{thm:main_theorem} to obtain the existence of a $T$-periodic solution.
\end{proof}

\section{An example} \label{sec:example}

We now construct a planar Hamiltonian system \eqref{eq:HS} with a blow-up solution, for which assumptions~\ref{hyp:A4} and \ref{hyp:H1} may hold or fail depending on the values of the parameters $\alpha$ and $\beta$. Precisely, working in canonical polar coordinates \eqref{eq:canonical-polar-coordinates}, we construct a Hamiltonian system admitting the solution $(\theta^*,\rho^*)\colon [0,T_\mathrm{F})\to \R\times (0,+\infty)$
\begin{align}\label{eq:example_solution}
\theta^*(t)\coloneqq \alpha (\sigma_{0}T_\mathrm{F})^{\alpha-1}\int_{0}^{t}(T_\mathrm{F}-s)^{-\beta\left(\alpha-1\right)}\,\dd s, &&\rho^*(t)\coloneqq\frac{\sigma_{0}T_\mathrm{F}}{(T_\mathrm{F}-t)^{\beta}}
\end{align}
where $T_\mathrm{F}\in(0,T]$ is the blow-up time,  while $\alpha>1$ and $\beta>0$ are parameters controlling the growth of the two coordinates. The parameter $\sigma_{0}$  controls the initial radius $\rho^*(0)=\sigma_0 T_\mathrm{F}^{1-\beta}$. We assume $\sigma_0\geq T_\mathrm{F}^{\beta-1}$ so that  $\rho^*(0)\geq 1$.

Denoting by $z^*_0$ the initial value of the blow-up solution \eqref{eq:example_solution}, we have two cases, according to the value of $\beta\left(\alpha-1\right)$:
\begin{itemize}
\item if $\beta\left(\alpha-1\right)\geq 1$, then $\displaystyle\lim_{t\to T_\mathrm{F}} \rot_t(z^*_0)=\displaystyle\lim_{t\to T_\mathrm{F}} \theta^*(t)=+\infty$\,; 
\item if $\beta\left(\alpha-1\right)< 1$, then 
\begin{equation*}
\lim_{t\to T_\mathrm{F}} \rot_t(z^*_0)=\frac{1}{2\pi}\lim_{t\to T_\mathrm{F}} \theta^*(t)=\frac{1}{2\pi} \frac{\alpha \sigma_0^{\alpha-1} }{1-\beta\left(\alpha-1\right)}T_\mathrm{F}^{\alpha-\beta\left(\alpha-1\right)} \,.
\end{equation*}
\end{itemize}

We construct the Hamiltonian $H$ as
\begin{equation}\label{eq:example_hamiltonian}
	H(t,\theta,\rho)=H_{0}(\rho)+K(t,\theta,\rho)\,,\qquad H_{0}(\rho)\coloneqq \rho^\alpha \,.
\end{equation}
The term $K$ will be defined to be always zero for $\rho<\frac{1}{2}$, so that $H$ can be suitably extended to the origin, where it would take value zero. Also notice that $H_0$ grows as $\frac{1}{2^\alpha}\abs{z}^{2\alpha}$,  since $\abs{z}=\sqrt{2\rho}$.

To define $K$, we set a small $\epsilon>0$ and introduce the two auxiliary functions $f(t,u),g(u)$ with the following properties:
\begin{itemize}
	\item[$f\colon$] let $f\colon [0,T_\mathrm{F})\times\R\to(-\epsilon,\epsilon)$ be $\mathcal{C}^\infty$-smooth, $2\pi$-periodic in the second variable $u$, satisfying $\norm{f}_\infty<\epsilon$ and, for every $t\in[0,T_\mathrm{F})$: \begin{align*}
	&\frac{\partial}{\partial u}f(t,0)=-\dot \rho^*(t)=-\frac{\sigma_0 T_\mathrm{F}\beta}{(T_\mathrm{F}-t)^{\beta+1}}\,,\\
	& -\dot \rho^*(t) \leq \frac{\partial}{\partial u} f(t,\cdot)<\frac{\beta}{T_\mathrm{F}}
		\end{align*}
	and $\supp f(t,\cdot)=[-1/2,1/2]+2\pi\Z$. Notice that $\dot \rho^*(t)\geq\dot \rho^*(0)\geq \frac{\beta}{T_\mathrm{F}}$ for every $t\in[0,T_\mathrm{F})$; 
	\item[$g\colon$] let $g\colon\R\to[0,1]$ be $\mathcal{C}^\infty$-smooth with $\supp g=[-1/2,1/2]$, $g(0)=1$, $g'(0)=0$ and  $\norm{g'}_{\infty}\leq L_{g}$ for some constant $L_g$.
\end{itemize}

We now set
\begin{equation}
	\label{eq:def_K}
	K(t,\theta,\rho)\coloneqq \begin{cases}f(t,\theta-\theta^*(t))\,g(\rho-\rho^*(t)) &\text{if $t\in[0,T_\mathrm{F})$}\\
		0 &\text{if $t\in[T_\mathrm{F},T)$}
	\end{cases}
\end{equation}
and extend it $T$-periodically to $t\in\R$.

First of all, we observe that \eqref{eq:example_solution} is a solution of \eqref{eq:ham-system-canonical-polar-coordinates} with Hamiltonian function \eqref{eq:example_hamiltonian}. Indeed, we have
\begin{align*}
	&\frac{\partial H}{\partial \theta}(t,\theta^*(t),\rho^*(t))=\frac{\partial K}{\partial \theta}(t,\theta^*(t),\rho^*(t))=-\dot\rho^*(t)\\
	&\frac{\partial H}{\partial \rho}(t,\theta^*(t),\rho^*(t))=\frac{\partial H_0}{\partial \rho}(t,\theta^*(t),\rho^*(t))=\alpha \rho^*(t)^{\alpha-1}=\dot\theta^*(t) \,.
\end{align*}
We will prove that, for the Hamiltonian system defined by \eqref{eq:example_hamiltonian} as above:
\begin{itemize}
\item	\ref{hyp:H0} and \ref{hyp:H2} hold for all $\alpha>1,\beta>0$ (cf.~Lemma~\ref{lem:exampleH0H2}); furthermore, also \ref{hyp:A1}, \ref{hyp:A2}, \ref{hyp:A3}, \ref{hyp:A5}, \ref{hyp:A6} are satisfied for every values $\alpha>1,\beta>0$ (cf.~Lemma~\ref{lem:exampleAminus4})  ; 
\item \ref{hyp:H1} holds if and only if $\beta(\alpha-2)\geq 1$ (cf.~Lemma~\ref{lem:exampleH1});
\item \ref{hyp:A4} holds if and only if $\beta(\alpha-1)\geq 1$ (cf.~Lemma~\ref{lem:exampleA4});
\end{itemize}

\begin{remark}
The last point gives a clear example of how \ref{hyp:A4}, which is always satisfied by the superlinear ODE \eqref{eq:superlinearODE}--\eqref{eq:superlinearODEgrowth}, might instead fail for general superquadratic Hamiltonian systems \eqref{eq:HS}, even when \ref{hyp:H0} and \ref{hyp:H2} hold. The possibility to drop assumption \ref{hyp:A4} from Theorem~\ref{thm:main_theorem} remains an open problem.
\end{remark}

\begin{remark}
The example also illustrates the advantage of the abstract rotational characterization given in Section~\ref{sec:main} with respect to pointwise estimates, such as \ref{hyp:H1} and \ref{hyp:H2}. While the latter two conditions might be easier to verify in some specific example, they might be too restrictive in other situations.
Here we showed it for \ref{hyp:H1}, but an analogous observation for the Ambrosetti--Rabinowitz condition \ref{hyp:H2} was done in \cite{Bos12}, and emerges already from the superlinear ODE \eqref{eq:superlinearODE}--\eqref{eq:superlinearODEgrowth}, which fails \ref{hyp:H2} but satisfies \ref{hyp:A4}, \ref{hyp:A5}, \ref{hyp:A6}.
\end{remark}

\begin{remark}
In the Hamiltonian system defined by \eqref{eq:example_hamiltonian}, there is a trivial $T$-periodic solution given by $z(t)\equiv 0$. We chose this structure to make clearer the role of the rotational assumptions at large radii. Notice however that the system can be easily modified to destroy such trivial $T$-periodic solution, without interfering with the behaviour at large.
Indeed, it is sufficient to add to the Hamiltonian \eqref{eq:example_hamiltonian} an additional perturbative $\mathcal{C}^\infty$-smooth term $K_0(t,z)$, $T$-periodic in time, with $\supp K_0(t,\cdot)=B[0,\sqrt{2\rho^*(0)-1}]$ for every $t$ and $\abs{\nabla_{z}K_0}\leq \frac{\sqrt{2\rho^*(0)-1}}{T}$, in the spirit of condition~\ref{cond:A3estimate}. Since this perturbation acts only on the region $\rho<\rho^*(0)-\frac{1}{2}$, all the lemmas above remain valid.
\end{remark}

\begin{lemma}\label{lem:exampleH0H2}
For every value $\alpha>1,\beta>0$ of the parameters, \ref{hyp:H0} and \ref{hyp:H2} hold for the Hamiltonian system defined by \eqref{eq:example_hamiltonian}.
\end{lemma}
\begin{proof}
By construction, $H(t,z)$ is continuously differentiable in $z$ and $F(t,z)\coloneqq J\nabla_{z}H(t,z)$ satisfies the Carathéodory conditions, is $T$-periodic in $t$ and is locally Lipschitz continuous in $z$, so conditions \ref{hyp:A1}, \ref{hyp:A2} are satisfied. Since $z=0$ is a constant trivial solution, \ref{hyp:A3} is trivially satisfied and therefore also \ref{hyp:H0} holds. 

Before we proceed, let us observe that for every $t\in[0,T_\mathrm{F})$:
\begin{gather}
	\abs{H(t,\theta,\rho)-H_{0}(\rho)}=\abs{K(t,\theta,\rho)}=\abs{f(t,\theta-\theta^*(t))\, g(\rho-\rho^*(t))}\leq \epsilon \notag\\[2mm]
	\abs{\frac{\partial K}{\partial{\rho}}(t,\theta,\rho)}\leq\abs{f(t,\theta-\theta^*(t)) \, g'(\rho-\rho^*(t))}\leq  \epsilon L_g \label{eq:example_estimate_dottheta}\\[2mm]
	-\dot \rho^*(t)\leq\frac{\partial H}{\partial{\theta}}(t,\theta,\rho)=\frac{\partial K}{\partial{\theta}}(t,\theta,\rho)=\frac{\partial f}{\partial \theta}(t,\theta-\theta^*(t)) \, g(\rho-\rho^*(t))\leq \frac{\beta}{T_\mathrm{F}} \,.\label{eq:example_estimate_dotrho}
\end{gather}
Hence 
\begin{equation} \label{eq:example_boundH}
	 \rho^{\alpha}-\epsilon=H_{0}(\rho)-\epsilon\leq H(t,\theta,\rho)\leq H_{0}(\rho)+\epsilon=\rho^{\alpha}+\epsilon
\end{equation}
 and moreover
\begin{equation} \label{eq:example_bound_dHdrho}
	\frac{\partial H}{\partial{\rho}}(t,\theta,\rho)\geq \alpha\rho^{\alpha-1}-\epsilon L_g.
\end{equation}
Take any $m>0$. By \eqref{eq:example_boundH}, taking $\rr_{1}\coloneqq (\epsilon+m)^{1/\alpha}$ we obtain 
\begin{equation*}
	H(t,\theta,\rho)>m,\quad \text{for every $\rho>\rr_{1}$, $\theta\in\R$ and a.e.~$t\in[0,T]$.}
\end{equation*}
Now, fix any $k\in\left(\frac{1}{2\alpha},\frac{1}{2}\right)$. As $\alpha>1$ and $2k\alpha-1>0$, there exists an $\rr_{2}>0$ such that 
\begin{equation}\label{eq:example_estimate1}
	(2k\alpha-1)\rho^{\alpha}-2k\epsilon L_g\rho>\epsilon\qquad\text{for every }\rho>\rr_{2}\,.
\end{equation}
Using, respectively, \eqref{eq:example_boundH}, \eqref{eq:example_estimate1} and \eqref{eq:example_bound_dHdrho}, we obtain  
\begin{equation}
	H(t,\theta,\rho)<\rho^{\alpha}+\epsilon<2k\rho(\alpha\rho^{\alpha-1}-\epsilon L_g)\leq 2k\rho\frac{\partial H}{\partial{\rho}}(t,\theta,\rho)
\end{equation}
for every $\rho>\rr_{2}$, $\theta\in\R$ and a.e.~$t\in[0,T]$. Taking $\rr>\max\{\rr_1,\rr_2\}$ gives~\ref{hyp:H2}.
\end{proof}

\begin{lemma}\label{lem:exampleAminus4}
	For every value $\alpha>1,\beta>0$ of the parameters, \ref{hyp:A1}, \ref{hyp:A2}, \ref{hyp:A3}, \ref{hyp:A5}, \ref{hyp:A6} hold for the Hamiltonian system defined by \eqref{eq:example_hamiltonian}.
\end{lemma}
\begin{proof}
\ref{hyp:A1}, \ref{hyp:A2}, \ref{hyp:A3} follow immediately from the validity of \ref{hyp:H0} in Lemma~\ref{lem:exampleH0H2}, and \ref{hyp:A6} by this and Lemma~\ref{lem:H2impliesA6}. 

The proof of \ref{hyp:A5} follows the same line as that of Lemma~\ref{lem:hs_proofA5}, with the following adaptation to remove \ref{hyp:H1} from the assumptions. In \eqref{eq:defin_Rn}, we require instead $R_n>\max\left\{ R_\delta,\tilde R_n, \sqrt{\tilde R_n^2 +\frac{2\beta T}{T_\mathrm{F}}}\right\}$.
By \eqref{eq:example_estimate_dotrho}, for $\zeta(t;0,z_0)\neq 0$, 
\begin{equation*}
\frac{\dd}{\dd t} \abs{\zeta(t;0,z_0)}=\frac{\dot \rho}{\sqrt{2\rho(t)}}=-\frac{1}{\sqrt{2\rho(t)}}\frac{\partial H}{\partial\theta}(t,\theta(t),\rho(t))\geq -\frac{\beta}{T_\mathrm{F} \sqrt{2\rho(t)}}=-\frac{\beta}{T_\mathrm{F}\abs{\zeta(t;0,z_0)}} \,.
\end{equation*}
Multiplying both sides by $\abs{\zeta(t;0,z_0)}$ and integrating over $[0,t]$, it follows that if $\abs{z_0}>R_n$ then  $\abs{\zeta(t;0,z_0)}>\tilde R_n$ for every $t\in[0,T]$ at which the solution is defined. Hence one of the first two alternatives in the proof of  Lemma~\ref{lem:hs_proofA5} is always satisfied and it is not necessary to consider the third alternative, which was the only case requiring \ref{hyp:H1}.
\end{proof}

\begin{lemma}\label{lem:exampleH1}
For the Hamiltonian system defined by \eqref{eq:example_hamiltonian}, \ref{hyp:H1} holds if and only if $\beta(\alpha-2) \geq 1$.
\end{lemma}
\begin{proof}
	First, we prove that there exists $c^*>0$, $\gamma \geq 1$  such that
	\begin{equation}\label{eq:example_H1star}
\dot{\theta}^*(t)\geq c^* \abs{\dot{\rho}^*(t)}^{\gamma} \qquad \text{for every $t\in[0,T_\mathrm{F})$}
	\end{equation} 
if and only if	$\beta(\alpha-2)\geq1$. Indeed, \eqref{eq:example_H1star} is equivalent to
 \begin{equation*}
 \frac{\alpha(\sigma_0 T_\mathrm{F})^{\alpha-1}}{(T_\mathrm{F}-t)^{\beta(\alpha-1)}}\geq c^* \left(\frac{\sigma_0 T_\mathrm{F}\beta}{(T_\mathrm{F}-t)^{\beta+1}}\right)^{\gamma} \,,
 \end{equation*}
 which can be satisfied by some sufficiently small $c^*>0$ if and only if the function $(T_\mathrm{F}-t)^{\beta(\alpha-1)-\gamma(\beta+1)}$ is bounded on $t\in[0,T_\mathrm{F})$, namely if and only if $\beta(\alpha-1)-\gamma(\beta+1)\geq 0$. Hence, a suitable $\gamma \geq 1$ exists if and only if 
	\begin{equation}
	\frac{\beta(\alpha-1)}{\beta+1}\geq \gamma \geq 1
\end{equation}
can be satisfied by some $\gamma$, namely if and only if $\beta(\alpha-2) \geq 1$.
This also shows that \ref{hyp:H1} is false for $\beta(\alpha-2) < 1$.
 
Let us now prove that if $\beta(\alpha-2) \geq 1$, then \ref{hyp:H1} holds, that is for arbitrary points $(t,\theta,\rho)$. Given $c^*>0$, $\gamma \geq 1$ for which \eqref{eq:example_H1star} holds, let us take any 
\begin{equation}\label{eq:example_H1constants}
 \rr>\max\left\{\frac{3}{2}, \left(\frac{2\epsilon L_g}{\alpha}\right)^\frac{1}{\alpha-1} \right\} \qquad\text{and}\qquad c=\frac{c^*}{2}\left(1-\frac{1}{2(\rr-1)}\right)^{\alpha-1}\,.
\end{equation}

First, we notice that \ref{hyp:H1} is trivially true if $\abs{\rho-\rho^*(t)}\geq \frac{1}{2}$, since the right-hand side is zero. It remains to prove the inequality on triplets $(t,\theta,\rho)$ such that $\abs{\rho-\rho^*(t)} <\frac{1}{2}$ and $\rho>\rr$.  On such points it holds:
\begin{equation}\label{eq:example_stima}
	\rho^{\alpha-1}\geq \left(\rho^*(t)-\frac{1}{2}\right)^{\alpha-1} \geq \left(1-\frac{1}{2(\rr-1)}\right)^{\alpha-1}\rho^*(t)^{\alpha-1} =\frac{2c}{c^*} \rho^*(t)^{\alpha-1} \,.
\end{equation}
Thus we obtain
\begin{align*} 
	\frac{\partial H}{\partial{\rho}}(t,\theta,\rho)
	&\stackrel{\eqref{eq:example_bound_dHdrho}}{\geq} \alpha\rho^{\alpha-1}-\epsilon  L_g 
	\geq \frac{\alpha}{2} \rho^{\alpha-1} + \frac{\alpha}{2} \rr^{\alpha-1} -\epsilon  L_g 
	 \stackrel{\eqref{eq:example_H1constants}}{\geq}  \frac{\alpha}{2} \rho^{\alpha-1}
	\\[2mm]
	&\stackrel{\eqref{eq:example_stima}}{\geq} \frac{c\alpha}{c^*}\rho^*(t)^{\alpha-1}
	=\frac{c}{c^*}\dot{\theta}^*(t)
	\stackrel{\eqref{eq:example_H1star}}{\geq} c \abs{\dot{\rho}^*(t)}^{\gamma}\\[2mm]
	&\stackrel{\eqref{eq:example_estimate_dotrho}}{\geq} c \abs{\frac{\partial H}{\partial{\theta}}(t,\theta,\rho)}^\gamma
\end{align*}
where in the last step we also used $\abs{\dot \rho^*}\geq \frac{\beta}{T_\mathrm{F}}$, which was observed when we defined~$f$.
	
\end{proof}

To study \ref{hyp:A4}, we require a preliminary lemma. 
\begin{lemma}
	\label{lem:blow-up-sol-for-example-system}
	Let $\beta(\alpha-1)\geq1$. A solution $(\theta(t),\rho(t))$ of the Hamiltonian system defined by \eqref{eq:example_hamiltonian} blows up in finite time within $(0,T]$  if and only if there exists $\tau\in[0,T_\mathrm{F})$ such that
\begin{equation}\label{eq:lemma_explosion_example}
\abs{\rho(t)-\rho^*(t)}<\frac{1}{2}\,,\qquad \text{for every $t\in[\tau,T_\mathrm{F})$\,.}
\end{equation}
Moreover, any such solution blows up exactly at time $T_\mathrm{F}$.
\end{lemma}

\begin{proof}
	By the first inequality in \eqref{eq:example_estimate_dotrho},  the function $\delta(t)\coloneqq\rho(t)-\rho^*(t)$ is non-increasing on any time interval $[0,\bar t)$ on which both $\rho(t)$ and $\rho^*(t)$ are defined. 
	
We first prove that every solution blowing up within $(0,T]$ must do it exactly at $t=T_\mathrm{F}$. Indeed no solution can blow up at a time in $(T_\mathrm{F},T]$, since the perturbation $K$ is zero in that time interval, hence the radius of each solution is constant. On the other hand, a solution blowing up at $\bar t \in(0,T_\mathrm{F})$ would imply $\delta (t)\to +\infty$ as $t\to \bar t^-$, contradicting the fact that $\delta$ is not-increasing.	
	
We can now prove the first part of the Lemma. If \eqref{eq:lemma_explosion_example} is satisfied for a certain value of $\tau$, then clearly $\rho(t)\to +\infty$ as $t\to T_\mathrm{F}^-$. 
	
Let us now assume that \eqref{eq:lemma_explosion_example} is false for every  $\tau\in[0,T_\mathrm{F})$.	
Since $\delta$ is defined on $[0,T_\mathrm{F})$ and non-increasing, there exists $\tilde t\in[0,T_\mathrm{F})$ such that $\abs{\delta (t)}\geq\frac{1}{2}$ for every $t\in [\tilde t,T_\mathrm{F})$. This means that $\dot \rho(t)=0$ for every $t\in [\tilde t,T_\mathrm{F})$ and therefore the solution is not blowing up at $T_\mathrm{F}$.
\end{proof}

\begin{lemma}\label{lem:exampleA4}
	For the Hamiltonian system defined by \eqref{eq:example_hamiltonian}, \ref{hyp:A4} holds if and only if $\beta(\alpha-1)\geq 1$.
\end{lemma}
\begin{proof}
	Let $(\theta(t),\rho(t))$ be a solution of the Hamiltonian system defined by \eqref{eq:example_hamiltonian} blowing up within $(0,T]$. Then, by Lemma \ref{lem:blow-up-sol-for-example-system}, we know that it blows up exactly at $T_\mathrm{F}$ and $\rho(t)-\rho^*(t)>-\frac{1}{2}$ for every $t\in[\tau,T_\mathrm{F})$.
	
	By \eqref{eq:example_estimate_dottheta} and since $\rho(t)\to+\infty$, there exists $\bar \tau\in[\tau,T_\mathrm{F})$ such that
	\begin{equation*}
		\dot{\theta}(t)\stackrel{\eqref{eq:example_bound_dHdrho}}{\geq} \alpha\rho(t)^{\alpha-1}-\epsilon L_g> \frac{\alpha}{2}\rho(t)^{\alpha-1}
		\geq \frac{\alpha}{2}\left(\rho^*(t)-\frac{1}{2}\right)^{\alpha-1}
		 \qquad\text{for every $t\in[\bar{\tau},T_\mathrm{F})$.}
	\end{equation*}
	Hence, naming $z_0$ the initial point of the solution, for any $t\in[\bar{\tau},T_\mathrm{F})$ it holds
		\begin{align*}
		\rot_{t}(z_{0})=\rot_{\bar \tau}(z_{0})+\frac{1}{2\pi}\int_{\bar \tau}^{t}\dot{\theta}(s)\dd s
		\geq \rot_{\bar \tau}(z_{0}) +\frac{1}{2\pi}\int_{\bar{\tau}}^{t}\frac{\alpha}{2}\left(\rho^*(s)-\frac{1}{2}\right)^{\alpha-1}\dd s \,.
	\end{align*}
Since $\rot_{\bar \tau}(z_{0})\in\R$, $\rot_{t}(z_{0})\to+\infty$ as $t\to T_\mathrm{F}$ if and only if $\int_{\bar \tau}^{t}\rho^*(s)^{\alpha-1}\dd s \to +\infty$. By construction, this happens if and only if $\beta(\alpha-1)\geq1$.

Hence, if $\beta(\alpha-1)\geq1$ then \ref{hyp:A4} holds. On the other hand, if $\beta(\alpha-1)<1$ then \ref{hyp:A4} is violated by the special solution $(\theta^*,\rho^*)$.
\end{proof}

\paragraph{Acknowledgements.}  A.C. and P.G. are members of the Gruppo Nazionale di Fisica Matematica of the Istituto Nazionale di Alta Matematica.

\end{document}